\documentclass[12pt]{article}
\usepackage{amssymb,amsfonts,amsmath,amsthm, breqn, color, float, mathtools, booktabs, hvfloat, url, enumitem}
\usepackage{adjustbox}
\usepackage{caption}
\newcommand{\cali}{{\mathcal I}}
\newcommand{\calj}{{\mathcal J}}

\newcommand{\beq}{\begin{eqnarray*}}
    \newcommand{\feq}{\end{eqnarray*}}
\newcommand{\beqn}{\begin{eqnarray}}
\newcommand{\feqn}{\end{eqnarray}}

\newtheorem{theorem}{Theorem}
\makeatletter \@addtoreset{theorem}{section}\makeatother

\newtheorem{definition}[theorem]{Definition}
\newtheorem{lemma}[theorem]{Lemma}

\newtheorem*{theorem*}{Theorem}
\newtheorem*{conj*}{Conjecture}

\newtheorem{proposition}[theorem]{Proposition}

\usepackage{algorithm}
\usepackage{graphicx}
\usepackage[noend]{algpseudocode}
\makeatletter
\def\BState{\State\hskip-\ALG@thistlm}
\makeatother

\DeclareCaptionLabelFormat{noname}{#2}
\newlength\myindent
\DeclareMathOperator{\Var}{Var}

\DeclareMathOperator{\I}{\textit{I}_{n}}
\DeclareMathOperator{\Lis}{\textit{L}}

\title{The longest increasing subsequence in involutions avoiding 3412 and another pattern}

\author{Toufik~Mansour\thanks{ Department of Mathematics, University of Haifa, 199 Abba Khoushy Ave, 3498838 Haifa, Israel;
\newline e-mail: tmansour@univ.haifa.ac.il} \and
Reza~Rastegar\thanks{Occidental Petroleum Corporation, Houston, TX 77046 and Departments of Mathematics and Engineering, University of Tulsa, OK 74104, USA - Adjunct Professor; e-mail:  reza\_rastegar2@oxy.com}
\and
Alexander~Roitershtein \thanks{Department of Statistics, Texas A\&M University, College Station, TX 77843, USA; \newline e-mail: alexander@stat.tamu.edu}
\and
G\"{o}khan Y\i ld\i r\i m\thanks{Department of Mathematics, Bilkent University, 06800 Ankara, Turkey;\newline e-mail: gokhan.yildirim@bilkent.edu.tr.
Corresponding author.}
}

\begin{document}
\maketitle

\begin{abstract}
In this note, we study the mean length of the longest increasing subsequence of a uniformly sampled involution that avoids the pattern $3412$ and another pattern.
\end{abstract}
\section{Introduction}
In this paper we study the longest increasing subsequence of involutions avoiding 3412 and another pattern. A permutation $\sigma=\sigma_1\sigma_2\cdots\sigma_n$ of length $n$ is defined as an  arrangement of the elements of the set
$[n]:=\{1,2,\cdots, n\}.$ A permutation $\sigma$ is called an \textit{involution} if $\sigma=\sigma^{-1},$ where $\sigma^{-1}_i=j$ if and only if $\sigma_j=i$. We use notations $S_n$ and $I_n$ to denote, respectively, the set of all permutations and the set of all involutions of length $n.$
A subsequence of $\sigma\in S_n$ is defined as a sequence $\sigma_{i_1}\sigma_{i_2}\cdots \sigma_{i_k}$ where $1\leq i_1<i_2<\cdots < i_k \leq n.$ The subsequence is called an increasing subsequence if $\sigma_{i_1} < \sigma_{i_2} < \cdots < \sigma_{i_k}.$
\par
For any permutation $\sigma,$ there is at least one longest increasing subsequence. We denote the length of this subsequence by $\Lis_n(\sigma).$
The celebrated Ulam's problem is concerned with the asymptotic behavior, as $n$ tends to infinity, of the expectation of $\Lis_n(\sigma)$ when $\sigma$ is chosen uniformly from $S_n$ \cite{Baik, Romik}. The classical Ulam's problem has been extended and generalized in various directions
\cite{Stanley, Stanleya}. In particular, asymptotic behavior of the distribution of the longest increasing subsequence of random involutions is the topic of \cite{Baik2, Kiwi}. 
\par
Variations of Ulam's problem have been considered also for permutations in $S_n$ avoiding certain patterns \cite{Deutsch1, mg, mg2, Reifegerste1}. For permutations $\pi=\pi_1\pi_2\cdots\pi_k\in S_k$ and $\sigma=\sigma_1\sigma_2\cdots\sigma_n\in S_n$, we say that
$\sigma$ contains \textit{pattern} $\pi$ if there exist $1\leq i_1<i_2<\cdots<i_k\leq n$ such that
\beq
\sigma_{i_s}<\sigma_{i_t} \quad \text{if and only if} \quad \pi_s<\pi_t \quad \mbox{for all} \quad 1\leq s,t\leq k.
\feq
For instance, the permutation $15243$ contains
$321$ as a pattern because it has the subsequences $*5*43,$ and $543$ matches the pattern $321.$
If $\sigma$ does not contain $\pi$ as a pattern, then we say that $\sigma$
{\it avoids} $\pi$ or $\sigma$ is a $\pi$-\textit{avoiding}
permutation. We denote by $S_n(\pi)$ and $\I(\pi),$ respectively, the sets of $\pi$-avoiding permutations and $\pi$-avoiding involutions of $[n].$ For more on the recent enumerative results on pattern-avoiding involutions, see \cite{BHV} and references therein, and for more on pattern-restricted permutation classes, see \cite{Vat}. 
\par
The goal of this paper is to study Ulam's problem in the context of involutions in $\I$ avoiding $3412$ and another pattern.
In \cite{Egge} Egge connected generating functions for various subsets of $\I(3412)$ with continued fractions and Chebyshev polynomials of the second kind, and gave a recursive formula for computing them. The formula exploits a bijection between $\I(3412)$ and Motzkin paths established in \cite{G}. Many of the results in \cite{Egge} are concerned with statistics of decreasing subsequences of involutions in $\I(3412).$ Later, Egge and Mansour \cite{EggMan} extended the results in \cite{Egge} to certain bivariate generating functions involving statistics of two-cycles in involutions. In this paper we extend
the method of \cite{Egge, EggMan} to certain bivariate generating functions involving the statistic $\Lis_n(\sigma),$ and use it as a tool for studying the Ulam's problem for such pattern-restricted involutions.  
\par
For a given set of patterns $T,$ let $\I(T)=\bigcap_{\tau\in T}\I(\tau)$ and denote by  $P_{n,T}$ the uniform distribution on $\I(T).$ Thus, the probability of choosing any $\sigma\in \I(T)$  under $P_{n,T}$ is $\frac{1}{|\I(T)|},$ where $|\,\cdot\,|$ is the size of the set. We use the notations $E_{n,T}(\,\cdot\,)$ and $\mbox{Var}_{n,T}(\,\cdot\,)$ to denote, respectively, the expectation and the variance operators under $P_{n,T}$.
We use the shortcut $\Lis_n$ to denote the random variable $\Lis_n(\sigma),$ where $\sigma\in S_n$ is a random permutation sampled uniformly from $I_n(T).$
\par
Throughout the paper, we write $a_n\sim b_n$  to indicate that $\lim_{n\to\infty}\frac{a_n}{b_n}=1.$ We have:

\begin{theorem}
\label{mn-thm1} Consider $\Lis_n$ on $\I(T)$ under the uniform probability measure. Then we have the following:
$\mbox{}$
\begin{enumerate}[label=(\roman*)]
\item If $T=\{3412\},$ then $E_{n,T}(\Lis_n)=\frac{4n}{9}.$
\item If $T=\{3412, 123\},$ then $E_{n,T} (\Lis_n)=\frac{n^2/2+3/4+(-1)^n/4}{n^2/4+7/8+(-1)^n/8}\sim2.$
\item If $T=\{3412, 213\}$ or $T=\{3412, 132\},$ then $E_{n,T}  (\Lis_n)\sim\frac{n}{\sqrt{5}}.$
\item If $T=\{3412, 321\},$ then $E_{n,T} (\Lis_n)\sim\frac{3+\sqrt{5}}{5+\sqrt{5}}n.$
\item If $T=\{3412, 123\cdots k\}$ for some $k\geq 1,$ then $E_{n,T} (\Lis_n) \sim k-1.$
\item If $T=\{3412, 4123\}$, then,
\beq
E_{n,T}(\Lis_n)\sim\frac{1}{457}(198\alpha^3-246\alpha^2-131\alpha+299)n\approx0.454689799955\cdots n.
\feq
Here $\alpha$ is the complex root of smallest absolute value of the polynomial $3x^4-3x^3-x^2+3x-1.$
\item If $T=\{3412, 4321\},$ then $E_{n,T} (\Lis_n)\sim\frac{5n}{8}$.
\end{enumerate}
\end{theorem}

Since $3412$ contains the patterns $231$ and $312$, we have
\beq
I_n(3412,231)=I_n(231) \qquad \mbox{\rm and} \qquad I_n(3412,312)=I_n(312).
\feq
As shown in section~3.2 of \cite{mg}, $E_{n,T}(\Lis_n)=\frac{n+1}{2}$
for $T=\{3412,231\}$ and $T=\{3412,312\}$. Thus, Theorem~\ref{mn-thm1} covers all possible cases for $\I(3412, \tau)$ with $\tau\in S_3$.

Using similar arguments we also obtained the asymptotic of $E_{n,T}(\Lis_n)$ and $\Var_{n,T}(\Lis_n)$ for all possible cases $\I(3412, \tau)$ with $\tau\in S_4$. We summarize these results in Table~\ref{tab:summary}, without explicit calculations for the sake of space.

\begin{table}[!t]
\centering
	\resizebox{\textwidth}{!}{%
		{\small\begin{tabular}{|l|l|l|}\hline
		&&\\
				$\tau$ & $H_\tau(x,q)=\sum_{n\geq0}\sum_{\sigma\in I_n(3412, \tau)}x^nq^{\Lis_n(\sigma)}$ &$E_{n, T}=E_{n, T}(\Lis_n),V_{n,T}=\Var_{n,T}(\Lis_n)$ for $T=\{3412,\tau\}$ \\[4pt]\hline\hline
				&&\\
				$1234$&$1+\frac{x}{(1-x)}q+\frac{x^2}{(1-x)^3(1+x)}q^2+\frac{x^3(x^2+1)}{(1-x)^5(x+1)^2}q^3$&$E_{n,T}\sim3$, $V_{n,T}\sim\frac{12}{n^2}$\\[4pt]\hline
				&&\\
				$1243,2134,1324$&$1+\frac{qx(x^4+(1+(q-2)x^2)(1-xq))}{(1-qx-x^2)^2(1-x)}$&$E_{n,T}\sim\frac{n}{\sqrt{5}}$, $V_{n,T}\sim\frac{4}{5\sqrt{5}}n$\\[4pt]\hline
				&&\\
				$1342,1423$&$\frac{(q-1)x^3+x^2+x-1}{x^3-x^2-(1+q)x+1}$&$E_{n,T}\sim\frac{(3-2\alpha)(\alpha+1)}{7}n$, $V_{n,T}\sim\frac{-7\alpha^2+5\alpha+10}{49}n$, where\\
				$2314,3124$&&$\alpha^3-\alpha^2-2\alpha+1=0,$ $\alpha\approx 0.44504$\\[4pt]\hline
				&&\\
				$1432,3214,2143,4231$&$\frac{1-x}{1-x-qx}$&$E_{n,T}\sim\frac{n}{2}$, $V_{n,T}\sim\frac{1}{4}n$\\[4pt]\hline
				&&\\
				$2341,4123$&$\frac{1}{1-\frac{xq}{1-x}-\frac{x^4q^2}{(1-x)^3(1+x)}}$&$E_{n,T}\sim\frac{198\alpha^3-246\alpha^2-131\alpha+299}{457}n$,\\
				&&$V_{n,T}\sim\frac{28800\alpha^3-7157\alpha^2-8959\alpha+47230}{208849}n$, where\\
				&&$3\alpha^4-3\alpha^3-\alpha^2+3\alpha-1=0,$ $\alpha\approx 0.45209$ \\[4pt]\hline
				&&\\
				$2413,3142$&$\frac{1-xq-x^2(q-1)-\sqrt{(1-xq-x^2(q-1))^2-4x^2}}{2x^2}$&$E_{n,T}\sim\frac{4n}{9}$, $V_{n,T}\sim\frac{4n}{27}$\\[4pt]\hline
				&&\\
				$2431,3241$&$\frac{1-qx-x^2}{q^2x^3+(q^2-q-1)x^2-2qx+1}$&$E_{n,T}\sim\frac{(\alpha+1)(\alpha+2)}{7}n$, $V_{n,T}\sim\frac{-7\alpha^2-4\alpha+13}{49}n$ where\\
				$4132,4213$&&$\alpha^3-\alpha^2-2\alpha+1=0,$ $\alpha\approx 0.44504$\\[4pt]\hline
				&&\\
				$3421,4312$&$\frac{1-(q+1)x}{(1-(q+1)x-qx^2)(1-qx)}$&$E_{n,T}\sim\frac{n}{2}$ $V_{n,T}\sim\frac{\sqrt{2}}{8}n$\\[4pt]\hline
				&&\\
				$4321$&$\frac{1-qx}{q(q-1)x^3+q(q-1)x^2-2qx+1}$&$E_{n,T}\sim\frac{5}{8}n$, $V_{n,T}\sim\frac{7}{64}n$\\[4pt]\hline
			
			\end{tabular}}}
			\caption{\label{tab:summary} The list of the generating functions and asymptotic values of the mean and variance of the length of the longest increasing subsequence for uniformly random involutions from $\I(3412,\tau)$ with $\tau \in S_4$.}
		\end{table}

\par
The rest of the paper is organized as follows. In Section~\ref{tzero} we consider $\I(3412)$ and prove part (i) of Theorem~\ref{mn-thm1}.
In Section~\ref{taus} we consider $\I(3412,\tau)$ with various patterns $\tau$ and prove the rest of Theorem~\ref{mn-thm1}.
\section{Longest increasing subsequences in $\I(3412)$}
\label{tzero}
For $\rho \in S_k$ and $\sigma \in S_m,$ we denote by $\rho \oplus \sigma$ their {\em direct sum}, which is a permutation in $S_{k+m}$ given by $\rho_1\cdots\rho_k(\sigma_1+k)\cdots(\sigma_m+k)$. Similarly, we denote by $\rho \ominus \sigma$ the {\em skew sum} of $\rho$ and $\sigma$, which is an element of $S_{k+m}$ given by $(\rho_1+m)\cdots(\rho_k+m)\sigma_1\cdots\sigma_m$.
\par
Our proofs make use of the following recursive structure of the involutions in $\I(3412)$, for the details see \cite[Remark~4.28]{G} and \cite[Proposition~2.9]{Egge}:
\begin{proposition}
\label{prop:permmap}
Let $\rho\in \I(3412)$. Then either
\begin{itemize}
\item[(i)] $\rho=1\oplus\rho'$ and $\rho'\in \mbox{I}_{n-1}(3412)$, or
\item[(ii)] $\rho=(1\ominus\rho''\ominus1)\oplus\rho'$, where $\rho''\in \mbox{I}_{m-2}(3412)$ and $\rho'\in \mbox{I}_{n-m}(3412)$
for some $m\geq2.$
\end{itemize}
\end{proposition}

\begin{proof}[Proof of Theorem~\ref{mn-thm1}-(i)]
Let $H(x,q)$ be the generating function for the number of
involutions in $\I(3412)$ according to the length of the longest increasing subsequence. More precisely,
\beqn
\label{h}
H(x,q)=\sum_{n\geq0}\sum_{\sigma\in
I_n(3412)}x^nq^{\Lis_n(\sigma)}.
\feqn
To obtain a closed form for $H(x,q),$ we partition $\I(3412)$ as a union of the following four non-overlapping subsets, by virtue of Proposition~\ref{prop:permmap}:
\begin{enumerate}[nosep,label=(\roman*)]
\item $\cali_{n,1}$ - the set of the empty involution;
\item $\cali_{n,2}$ - the set of the involutions in $\I(3412)$ that start with $1;$
\item $\cali_{n,3}$ - the set of the involutions in $\I(3412)$ that start with $21;$
\item $\cali_{n,4}$ - the set of the involutions in $\I(3412)$ that can be written as $(1\ominus\sigma''\ominus 1)\oplus\sigma'$, where $\sigma''$ is a nonempty $3412$-avoiding involution and $\sigma'$ is any $3412$-avoiding involution.
\end{enumerate}
Adding together contributions of all the four sets, we obtain:
\beq
H(x,q)=
\underbrace{1}_{\cali_{n,1}}
+\underbrace{xqH(x,q)}_{\cali_{n,2}}
+\underbrace{x^2qH(x,q)}_{\cali_{n,3}}
+\underbrace{x^2(H(x,q)-1)H(x,q)}_{\cali_{n,4}}.
\feq
Hence,
\beq
H(x,q)=\frac{1-xq-x^2(q-1)-\sqrt{(1-xq-x^2(q-1))^2-4x^2}}{2x^2}.
\feq
Note that $H(x,1)=\frac{1-x-\sqrt{1-2x-3x^2}}{2x^2}$, which is the generating function for Motzkin numbers
\cite{Egge, G}. Furthermore,
\beq
\frac{\partial}{\partial q}H(x,q)\,\Big|_{q=1}
=-\frac{x+1}{2x}+\frac{1+x^2}{2x\sqrt{1-2x-3x^2}}.
\feq
Hence,
\beq
E_{n,3412}(\Lis_n)=\frac{[x^n]\frac{\partial}{\partial
q}H(x,q)\mid_{q=1}}{[x^n]H(x,1)}\sim\frac{\frac{2n\sqrt{3}}{9\sqrt{\pi
n}n}3^{n+1}}{\frac{\sqrt{3}}{2\sqrt{\pi
n}n}3^{n+1}}=\frac{4n}{9},
\feq
which completes the proof of Theorem~\ref{mn-thm1}-(i).
\end{proof}
\section{Longest increasing subsequences in $\I(3412, \tau)$}
\label{taus}
In this section, we extend our arguments from $\I(3412)$ to $\I(3412, \tau)$ for various patterns $\tau.$ Toward this end, similar to \eqref{h}, we define
\beq
H_\tau(x,q)=\sum_{n\geq0}\,\sum_{\sigma\in I_n(3412, \tau)}x^nq^{\Lis_n(\sigma)}.
\feq
More generally, for a collection of patterns $T$, we set
\beq
H_T(x,q)=\sum_{n\geq0}\,\,\sum_{\sigma\in I_n(3412)\bigcap I_n(T)}x^nq^{\Lis_n(\sigma)}.
\feq
When $T=\{\tau,\tau'\}$, for simplicity, we write $H_{\tau,\tau'}(x,q)$. We also set $H_\emptyset(x,q):=0$ and let $H_{\tau\slash\tau'}(x,q):=H_\tau(x,q)-H_{\tau,\tau'}(x,q)$ denote the corresponding generating function for the involutions in $I_n(3412,\tau)$ that contain the pattern $\tau'$.

We call a permutation {\it irreducible} if it cannot be represented as a direct sum of two nonempty permutations. It is easy to show that every permutation $\rho$ can be written as a direct sum
\beq
\rho=\rho^{(1)}\oplus\rho^{(2)}\oplus\cdots\oplus\rho^{(k)},
\feq
where $\rho^{(1)}, \ldots, \rho^{(k)}$ are nonempty irreducible permutations, uniquely determined by $\rho.$ We next introduce a \textit{bar operator} for permutations following \cite{Egge}.
\begin{definition}
For $\rho\in S_m,$ define $\overline{\rho}$ as follows:
\begin{enumerate}
\item
$\overline{\emptyset} = \emptyset$ and $\overline{1} = \emptyset$.
\item
If $m\ge 2$ and there exists a permutation $\sigma$ such that $\rho = 1 \ominus \sigma \ominus 1,$ then $\overline{\rho} = \sigma$.
\item
If $m\ge 2$ and there exists a permutation $\sigma$ such that $\rho = 1 \ominus \sigma$, and $\sigma$ does not end with 1, then $\overline{\rho} = \sigma$.
\item
If $m\ge 2$ and there exists a permutation $\sigma$ such that $\rho = \sigma \ominus 1$, and $\rho$ does not begin with $m,$ then $\overline{\rho} = \sigma$.
\item
If $m\ge 2$ and $\rho$ does not begin with $m$, and it does not end with 1, then $\overline{\rho} = \rho$.
\end{enumerate}
\end{definition}
Our main technical tool for calculating the corresponding generating functions for the classes $\I(3412,\tau)$ is the following extension of a result for $\I(3412)$ given by Corollary~5.6 in
\cite{Egge}.
\begin{proposition}
\label{mth}
Suppose that $\tau=\tau^{(1)}\oplus\tau^{(2)}\oplus\cdots\oplus\tau^{(s)}$ is a direct sum of nonempty irreducible permutations $\tau^{(1)}, \ldots, \tau^{(s)}$ such that $\tau^{(1)}$ is not a decreasing sequence. For $i\in [s],$ define
$$\theta^{(i)}:=\overline{\tau^{(1)}\oplus\cdots\oplus\tau^{(i)}}\qquad  \mbox{ and } \qquad \theta^{<i>}:=\tau^{(i)}\oplus\cdots\oplus\tau^{(s)}.$$ Then we have:
\begin{itemize}
\item[(i)] If $\tau^{(1)}=1,$ then
\begin{align*}
H_\tau(x,q)&=1+\frac{xq}{1-x}H_{\theta^{(<2>)}}(x,q)\\
&+x^2\sum_{i=2}^s\big\{H_{\theta^{(i)}\slash12}(x,q)-H_{\theta^{(i-1)}\slash12}(x,q)\big\}
H_{\theta^{<i>}}(x,q).
\end{align*}
\item[(ii)] If $\tau^{(1)}=21,$ then
\begin{align*}
H_\tau(x,q)&=1+xqH_{\rho}(x,q)+\frac{x^2q}{1-x}H_{\theta^{<2>}}(x,q)\\
&+x^2\sum_{i=2}^s\big\{H_{\theta^{(i)}\slash12}(x,q)-\delta_{i>2}H_{\theta^{(i-1)}\slash12}(x,q)\big\}
H_{\theta^{<i>}}(x,q),
\end{align*}
where $\delta_{A}$ is one if $A$ is true, and is zero otherwise.
\item[(iii)] If $\tau^{(1)}=m(m-1)\cdots1$ with $m\geq3,$ then
\begin{align*}
H_\tau(x,q)&=1+(x+x^2+\cdots+x^{m-1})qH_{\rho}(x,q)+\frac{x^mq}{1-x}H_{\theta^{<2>}}(x,q)\\
&+x^2\sum_{i=1}^s\big\{H_{\theta^{(i)}\slash 12}(x,q)-H_{\theta^{(i-1)}\slash 12}(x,q)\big\}
H_{\theta^{<i>}}(x,q).
\end{align*}
\item[(iv)] If $\tau^{(1)}\neq m(m-1)\cdots1$ and $\rho^{(1)}\in S_m$ with $m\geq3,$ then
\begin{align*}
H_\tau(x,q)&=1+\frac{xq}{1-x}H_{\rho}(x,q)\\
&+x^2\sum_{i=1}^s\big\{H_{\theta^{(i)}\slash 12}(x,q)-H_{\theta^{(i-1)}\slash 12}(x,q)\big\}
H_{\theta^{<i>}}(x,q).
\end{align*}
\end{itemize}
\end{proposition}
We will only prove parts (i) and (iv) of the proposition. The proofs of the other two cases are very similar, and therefore are omitted.
\begin{proof}[Proof of Proposition~\ref{mth}-(i)]
Assume first that $\tau^{(1)}=1$. We partition the set $I_n(3412,\tau)$ into three non-overlapping subsets:
\begin{enumerate}[nosep,label=(\roman*)]
\item $\calj_{n,1}$ - the set of the empty involution;
\item $\calj_{n,2}$ - the set of those involutions of the form  $r(r-1)\cdots1\oplus\sigma'$ for some $r\geq 2;$
\item $\calj_{n,3}$ - the set of those involutions which do not begin with a decreasing sequence.
\end{enumerate}
It is easy to see that the involutions in the sets $\calj_{n,1}$ and $\calj_{n,2}$ contribute $1$ and $\frac{xq}{1-x}H_\tau(x,y),$ respectively, to $H_\tau(x,y).$ To obtain the contribution of the involutions in the set $\calj_{n,3},$ we first observe that in view of Proposition \ref{prop:permmap}, all involutions in $\calj_{n,3}$ can be written in the form $\sigma=(1\ominus\sigma''\ominus1)\oplus\sigma'$ with $\sigma''$ that contains $12$. Thus, the involutions in $\calj_{n,3}$ that avoid $\tau^{(1)}$ contribute $x^2H_{\theta^{(1)}\slash 12}(x,q)H_{\tau}(x,q)=0.$ Furthermore, any involution in $\calj_{n,3}$ that contains $\tau^{(1)},$ avoids $\theta^{(i)}$ and contains $\theta^{(i-1)}$ for some $i=2,3,\ldots,s.$ The total contribution of such involutions into $H_\tau(x,q)$ is equl to
\beq
x^2\sum_{i=2}^s \left(H_{\theta^{(i)}\slash 12}(x,q) -H_{\theta^{(i-1)}\slash 12}(x,q)\right) H_{\theta^{<i>}}(x,q).
\feq
Adding together the contributions of $\calj_{n,1},$ $\calj_{n,2},$ and $\calj_{n,3},$ we obtain the desired result.
\end{proof}
\begin{proof}[Proof of Proposition~\ref{mth}-(iv)]
Suppose now that $\tau^{(1)}\neq m(m-1)\cdots1$ and $\tau^{(1)}\in S_m$ with $m\geq3.$
We will consider again the partition $I_n(3412,\tau)=\bigcup_{k=1}^3 \calj_{n,k}$
defined in the course of the proof of part (i) of the proposition. It is easy to verify that in this case,  $\calj_{n,1}$ contributes $1$ to $H_\tau(x,q),$ while permutations in the set $\calj_{n,2}$ contribute $\frac{xq}{1-x}H_\tau(x,y).$ To obtain the contribution of $\calj_{n,3},$ recall that by Proposition~\ref{prop:permmap}, all involutions in this set have the form $\sigma=(1\ominus\sigma''\ominus1)\oplus\sigma'$ where $\sigma''$ contains $12.$ Thus, the involutions in $\calj_{n,3}$ that avoid $\tau^{(1)}$ contribute $x^2H_{\theta^{(1)}\slash 12}(x,q)H_{\tau}(x,q),$  while the involutions in $\calj_{n,3}$ that contain $\tau^{(1)}$ contribute
\beq
x^2\sum_{i=1}^s \left(H_{\theta^{(i)}\slash 12}(x,q) -H_{\theta^{(i-1)}\slash 12}(x,q)\right) H_{\theta^{<i>}}(x,q).
\feq
Adding up all the contributing terms listed above, yields the desired result.
\end{proof}
The rest of this section is divided into fives subsections, each one is concerned with $\I(3412,\tau)$ for a particular type of pattern $\tau$ and presents the proof of the corresponding part in Theorem~\ref{mn-thm1}.
\subsection{$E_{n,T}(\Lis_n)$ on $\I(3412, \tau)$ with $\tau\in S_2$}
Note that the only involution in $\I(3412,12)$ is $n(n-1)\cdots1.$ Thus,
\beqn
\label{a}
H_{12}(x,q)=1+\frac{xq}{1-x}.
\feqn
Similarly, the only involution in $\I(3412,21)$ is $12\cdots n.$ . Thus,
\beq
H_{21}(x,q)=\frac{1}{1-xq}.
\feq
\subsection{$E_{n,T}(\Lis_n)$ on $\I(3412, \tau)$ with $\tau\in S_3$}
\begin{proof}[Proof of Theorem~\ref{mn-thm1}-(ii)]
An application of Proposition~\ref{mth}-(i) with $\tau=1\oplus1\oplus1=123$ gives
\begin{align*}
H_{123}(x,q)&=1+\frac{xq}{1-x}H_{12}(x,q)+x^2(H_{12\slash 12}(x,q)-H_{1\slash 12}(x,q))H_{12}(x,q)\\
&+x^2(H_{123\slash 12}(x,q)-H_{12\slash 12}(x,q))H_1(x,q).
\end{align*}
It follows from \eqref{a} and the decomposition
\beqn
\label{fact}
H_{123\slash 12}(x,q)=H_{123}(x,q)-H_{12}(x,q)
\feqn
that
\begin{align*}
H_{123}(x,q)&=1+\frac{xq}{1-x}\left(1+\frac{xq}{1-x}\right)+x^2H_{123}(x,q)-x^2\left(1+\frac{xq}{1-x}\right).
\end{align*}
Therefore,
\beq
H_{123}(x,q)=1+\frac{xq(1-x(1-q)-x^2+x^3)}{(1-x)^3(1+x)}.
\feq
Hence, for $T=\{3412, 123\}$ we have:
\beq
E_{n,T}(\Lis_n)=\frac{[x^n]\frac{\partial}{\partial q}H_{123}(x,q)\mid_{q=1}}{[x^n]H_{123}(x,1)}
=\frac{n^2/2+3/4+(-1)^n/4}{n^2/4+7/8+(-1)^n/8}\sim2.
\feq
\end{proof}

\begin{proof}[Proof of Theorem~\ref{mn-thm1}-(iii)]
Proposition~\ref{mth}-(i) implies that for $\tau=1\oplus21=132,$
\beq
H_{132}(x,q)=1+\frac{xq}{1-x}H_{21}(x,q)+x^2\big(H_{132\slash 12}(x,q)-H_{1\slash 12}(x,q)\big)H_{21}(x,q).
\feq
Using \eqref{fact} and the fact that $H_{1\slash 12}(x,q)=0,$ we get
\beq
H_{132}(x,q)=\frac{1-x^2(1-q)}{1-xq-x^2}.
\feq
Therefore, for $T=\{3412, 132\}$ we have:
\beq
E_{n,T}(\Lis_n)\sim\frac{n}{\sqrt{5}}.
\feq
We next apply Proposition~\ref{mth}-(ii) to $\tau=21\oplus1=213,$ to get
\beq
H_{213}(x,q)=1+xqH_{213}(x,q)+\frac{x^2q}{1-x}+x^2\big(H_{213}(x,q)-H_{12}(x,q)\big)H_1(x,q).
\feq
It follows then from \eqref{a} that
\beq
H_{213}(x,q)=\frac{1-x^2(1-q)}{1-xq-x^2}.
\feq
Hence, for $T=\{3412,213\}$ we have: $E_{n,T}(\Lis_n)\sim\frac{n}{\sqrt{5}}.$
\end{proof}

\begin{proof}[Proof of Theorem~\ref{mn-thm1}-(iv)]
An application of Proposition~\ref{mth}-(iii) to $\tau=321$ yields:
\beq
H_{321}(x,q)&=1+(x+x^2)qH_{321}(x,q)+x^2(H_{1}(x,q)-H_{1,12}(x,q))H_{321}(x,q).
\feq
Since $H_1(x,q)=H_{1,12}(x,q)=1,$ this implies that
\beq
H_{321}(x,q)=\frac{1}{1-qx-qx^2}.
\feq
Thus, for $T=\{3412, 321\}$ we have $E_{n,T}(\Lis_n)\sim\frac{3+\sqrt{5}}{5+\sqrt{5}}n.$
\end{proof}

\subsection{$E_{n,T}(\Lis_n)$ on $\I(3412, \tau)$ with $\tau=12\cdots k$}
\begin{proof}[Proof of Theorem~\ref{mn-thm1}-(v)]
Let $F_k(x,q):=H_{12\cdots k}(x,q)$. Applying Proposition~\ref{mth} to the permutation $\tau=12\cdots k$ with $k\geq1$, we obtain:
\beq
F_k(x,q)=1+\frac{xq}{1-x}F_{k-1}(x,q)+x^2\sum_{i=3}^k(F_i(x,q)-F_{i-1}(x,q))F_{k-i+1}(x,q).
\feq
Let $F(x,q;y):=\sum_{k\geq1}F_k(x,q)y^k$. Multiplying both sides of the above recurrence equation by $y^k,$ summing over $k\geq1,$ and using the fact that $F_0(x,q)=0$ and $F_1(x,q)=1$, we obtain:
\begin{align*}
F(x,q;y)&=\frac{y}{1-y}+\frac{xqy}{1-x}F(x,q;y)
+\frac{x^2}{y}F(x,q;y)F(x,q;y)-x^2yH_{12}(x,q)F(x,q;y)\\
&-x^2F(x,q;y)-x^2F(x,q;y)F(x,q;y)+x^2yF(x,q;y).
\end{align*}
Taking \eqref{a} into account and solving for $F(x,y;q),$ we obtain:
\begin{align*}
F(x,q;y)&=\frac{y}{1-y}+\frac{(1-qyx-(1+qy)x^2-\sqrt{(1-qyx-(qy+1)x^2)^2-4q(1+x)x^3y})y}{2x^2(1-y)}\\
&=\frac{y}{1-y}+\frac{qxy^2}{(1-y)(1-x-qyx)}C\left(\frac{qx^3y}{(1+x)(1-x-qyx)^2}\right),
\end{align*}
where $C(x)=\frac{1-\sqrt{1-4x}}{2x}$ is the generating function for the Catalan numbers $c_n=\frac{1}{n+1}\binom{2n}{n}$.
\par
Substituting a series representation of the generating function from A001263 in \cite{Slo}, we obtain that
\beq
F(x,q;y)=\frac{y}{1-y}+\frac{1}{1-y}\sum_{j\geq0}\sum_{i=1}^{j+1}
\frac{\frac{1}{i}\binom{j-1}{i-1}\binom{j}{i-1}x^{2i+j-1}}{(1-x)^{2j+1}(1+x)^j}q^{j+1}y^{j+2}.
\feq
Therefore, for $k\geq2$ we have:
\beqn
\label{eqyFk}
[y^k]F(x,q;y)=1+\sum_{j=0}^{k-2}\sum_{i=1}^{j+1}
\frac{\frac{1}{i}\binom{j-1}{i-1}\binom{j}{i-1}x^{2i+j-1}}{(1-x)^{2j+1}(1+x)^j}q^{j+1}.
\feqn
Hence, for all $k\geq 2,$ using the usual bracket notation for coefficient extraction,
\beq
[x^ny^k]F(x,1;y)\sim\frac{1}{(k-1)2^{k-2}(2k-4)!}\binom{2k-4}{k-2}n^{2k-4}
\feq
and
\beq
[x^ny^k]\frac{\partial}{\partial q}F(x,q;y)\,\Big|_{q=1}\sim\frac{1}{2^{k-2}(2k-4)!}\binom{2k-4}{k-2}n^{2k-4},
\feq
which yields the result in Theorem~\ref{mn-thm1}-(v).
\end{proof}
\subsection{$E_{n,T}(\Lis_n)$ on $\I(3412, \tau)$ with $\tau=k12\cdots(k-1)$}
\begin{proof}[Proof of Theorem~\ref{mn-thm1}-(vi)]
Let $G_k(x,q):=H_{k12\cdots(k-1)}(x,q)$. Applying Proposition~\ref{mth}-(iv) to $\tau=k12\cdots(k-1)$ with $k\geq3,$ we obtain
that
\beq
G_k(x,q)=1+\frac{xq}{1-x}G_k(x,q)+x^2(F_{k-1}(x,q)-F_2(x,q))G_k(x,q),
\feq
which in view of \eqref{a} leads to
$$G_k(x,q)=\frac{1}{1-\frac{xq}{1-x}-x^2(F_{k-1}(x,q)-1-\frac{xq}{1-x})}.$$
Taking \eqref{eqyFk} into account, we arrive to the following result:
\begin{lemma}
For $k\geq 3,$
\beq
H_{k12\cdots(k-1)}(x,q)=\frac{1}{1-\frac{xq}{1-x}-x^2\sum_{j=1}^{k-3}\sum_{i=1}^{j+1}
\frac{\frac{1}{i}\binom{j-1}{i-1}\binom{j}{i-1}x^{2i+j-1}}{(1-x)^{2j+1}(1+x)^j}q^{j+1}}.
\feq
\end{lemma}
For example, $H_{4123}(x,q)=\frac{1}{1-xq/(1-x)-x^4q^2/((1-x)^3(1+x))}$. Let $\alpha$ be the root of smallest absolute value of the polynomial $3x^4-3x^3-x^2+3x-1$. Thus $\alpha\approx0.45208778430$, and for $T=\{3412,4123\}$ we have:
\begin{align*}
E_{n,T}(\Lis_n)&=\frac{[x^n]\frac{\partial}{\partial q}H_{4123}(x,q)\mid_{q=1}}{[x^n]H_{4123}(x,1)}\\
&\sim\frac{1}{457}(198\alpha^3-246\alpha^2-131\alpha+299)n\approx0.454689799955\cdots n.
\end{align*}
This completes the proof of Theorem~\ref{mn-thm1}-(vi).
\end{proof}
\subsection{$E_{n,T}(\Lis_n)$ on $\I(3412, \tau)$ with $\tau=k(k-1)\cdots1$}
\begin{proof}[Proof of Theorem~\ref{mn-thm1}-(vii)]
Let $F_k(x,q):=H_{k(k-1)\cdots1}(x,q)$. Applying Proposition~\ref{mth} to the permutation $\tau=k(k-1)\cdots1$ with $k\geq3,$ we see that
\begin{align*}
F_k(x,q)&=1+(x+x^2+\cdots+x^{k-1})qF_k(x,q)\\
&+x^2(F_{k-2}(x,q)-1-(x+x^2+\cdots+x^{k-3})q)F_k(x,q).
\end{align*}
Thus,
\beqn
\label{d}
F_k(x,q)&=\frac{1}{1-qx-(q-1)x^2-x^2F_{k-2}(x,q)}
\feqn
with $F_1(x,q)=1$ and $F_2(x,q)=\frac{1}{1-qx}.$ Iterating this equation, one can obtain an expression for $F_k(x,q)$ in the form of
finite continued fractions. Alternatively, $F_k(x,q)$ can be expressed in terms of Chebyshev polynomials.
\par
Recall that Chebyshev polynomials of the second kind can be defined as the solution to the recursion
\beq
U_n(t)=2tU_{n-1}(t)-U_{n-2}(t)
\feq
with initial conditions $U_0(t)=1$ and $U_1(t)=2t.$ Using this recursion and induction, one can derive the following result from \eqref{d}.
\begin{lemma}
\label{lem}
For all $k\geq1,$
\beq
H_{(2k+1)(2k)\cdots1}(x,q)=\frac{U_{k-1}\left(\frac{1-qx-(q-1)x^2}{2x}\right)
-xU_{k-2}\left(\frac{1-qx-(q-1)x^2}{2x}\right)}
{x\left(U_{k}\left(\frac{1-qx-(q-1)x^2}{2x}\right)-xU_{k-1}\left(\frac{1-qx-(q-1)x^2}{2x}\right)\right)}
\feq
and
\beq
H_{(2k+2)(2k+1)\cdots1}(x,q)=\frac{\frac{1-xq}{x}U_{k-1}\left(\frac{1-qx-(q-1)x^2}{2x}\right)
-U_{k-2}\left(\frac{1-qx-(q-1)x^2}{2x}\right)}
{x\left(\frac{1-xq}{x}U_{k}\left(\frac{1-qx-(q-1)x^2}{2x}\right)-U_{k-1}\left(\frac{1-qx-(q-1)x^2}{2x}\right)\right)}.
\feq
\end{lemma}
We remark that the results in Lemma~\ref{lem} with $q=1$ recover formulas (7) and (8) in \cite{Egge} for ordinary generating functions
for the number of involutions avoiding $3412$ and $k(k-1)\cdots1.$
\par
An application of the lemma with $k=1$ yields for $T=\{3412,4321\}:$
\beq
E_{n,T}(\Lis_n)&=\frac{[x^n]\frac{\partial}{\partial q}H_{4321}(x,q)\mid_{q=1}}{[x^n]H_{4321}(x,1)}
\sim\frac{5}{8}n,
\feq
which completes the proof of Theorem~\ref{mn-thm1}-(vii).
\end{proof}

\section*{Acknowledgement}
G. Y\i ld\i r\i m was partially supported by Tubitak-Bideb 2232 Grant no: 118C029.


\end{document}